\documentclass[11pt]{amsproc}
\usepackage{amssymb,amsthm,anysize,enumerate,float,times,amsmath,hyperref,url}
\hypersetup{citecolor=red, linkcolor=blue, colorlinks=true}

\usepackage{amssymb,latexsym,verbatim,pstricks,tikz} % verbatim for ``comment''
\usepackage{hyperref,url,pst-node}
\usetikzlibrary{calc}
\hypersetup{citecolor=red, linkcolor=blue, colorlinks=true}

\theoremstyle{plain}
\newtheorem{theorem}{Theorem}[section]
\newtheorem{corollary}[theorem]{Corollary}

\newtheorem{lemma}[theorem]{Lemma}

\newtheorem{hypothesis}[theorem]{Hypothesis}
\newtheorem{definition}[theorem]{Definition}

\theoremstyle{remark}
\newtheorem{remark}[theorem]{Remark}

\newcommand{\PG}{\textnormal{PG}}

\renewcommand\le{\leqslant}
\renewcommand\ge{\geqslant}

\title[Point-primitive generalised hexagons and octagons]{Point-primitive generalised hexagons and octagons}

\author[J.~Bamberg et al.]{John Bamberg, S. P. Glasby$^*$, Tomasz Popiel, Cheryl E. Praeger$^\dag$}
\address{
Centre for the Mathematics of Symmetry and Computation\\
School of Mathematics and Statistics\\
The University of Western Australia\\
35 Stirling Highway, Crawley, WA 6009, Australia.
\newline $^*$Also affiliated with The
Department of Mathematics, University of Canberra, Australia. 
\newline $^\dag$ Also affiliated with King Abdulaziz University, Jeddah, Saudi Arabia.
}
\email{\{john.bamberg,stephen.glasby,tomasz.popiel,cheryl.praeger\}@uwa.edu.au}
\thanks{The first author acknowledges the support of the Australian Research Council (ARC) Future Fellowship FT120100036. 
The second and third authors acknowledge the support of the ARC Discovery Grants DP130100106 and DP1401000416, respectively. 
The research reported in the paper forms part of the ARC Discovery Grant DP1401000416 of the second and fourth authors. 
The fifth author is grateful to the Centre for the Mathematics of Symmetry and Computation (UWA) for its hospitality during his visit in July of 2014, and acknowledges the support of the ARC Discovery Grant DP1401000416 for funding this visit. 
His work was also supported by the research projects 302660/2013-5 (CNPq, Produtividade em Pesquisa), 475399/2013-7 (CNPq, Universal), and APQ-00452-13 (Fapemig, Universal).}

\author[]{Csaba Schneider}
\address{
Departamento de Matem\'{a}tica, 
Instituto de Ci\^{e}ncias Exatas, 
Universidade Federal de Minas Gerais, 
Av.~Ant\^{o}nio Carlos, 6627, 31270-901,
Belo Horizonte, MG, Brazil. 
}
\email{csaba@mat.ufmg.br}


\begin{document}

\begin{abstract} 
In 2008, Schneider and Van Maldeghem proved that if a group acts flag-transitively, point-primitively, and line-primitively on a generalised hexagon or generalised octagon, then it is an almost simple group of Lie type. 
We show that point-primitivity is sufficient for the same conclusion, regardless of the action on lines or flags. 
This result narrows the search for generalised hexagons or octagons with point- or line-primitive collineation groups beyond the classical examples, namely the two generalised hexagons and one generalised octagon admitting the Lie type groups $\mathsf{G}_2(q)$, $\,^3\mathsf{D}_4(q)$, and $\,^2\mathsf{F}_4(q)$, respectively. 
\end{abstract}

\maketitle

\section{Introduction}\label{section:introduction}

Generalised polygons were introduced by Tits~\cite{Tits:1959cl} in an attempt to find geometric models for simple groups of Lie type. 
In particular, the group $\mathsf{PSL}(3,q)$ is admitted by the Desarguesian projective plane $\PG(2,q)$; the groups $\mathsf{PSp}(4,q)$, $\mathsf{PSU}(4,q)$, $\mathsf{PSU}(5,q)$ are admitted by certain generalised quadrangles; and $\mathsf{G}_2(q)$, $\,^3\mathsf{D}_4(q)$, $\,^2\mathsf{F}_4(q)$ arise as automorphism groups of two generalised hexagons and a generalised octagon, respectively (up to point--line duality).
These generalised polygons are called the \emph{classical} generalised polygons~\cite[Chapter~2]{Van-Maldeghem:1998sj}, and they serve as the examples that have the greatest degree of symmetry: their automorphism groups act primitively and distance-transitively on both points and lines. 
Buekenhout and Van Maldeghem~\cite{Buekenhout:1994zp} showed that the stronger of these symmetry conditions, distance-transitivity on points and lines, characterises the classical generalised polygons. 
However, it is not yet known whether there are non-classical generalised polygons having an automorphism group acting primitively on points or on lines. 
We make progress towards resolving this question by showing that if a generalised hexagon or generalised octagon admits an automorphism group $G$ that acts primitively on points, then $G$ must be an almost simple group of Lie type, regardless of its action on lines.

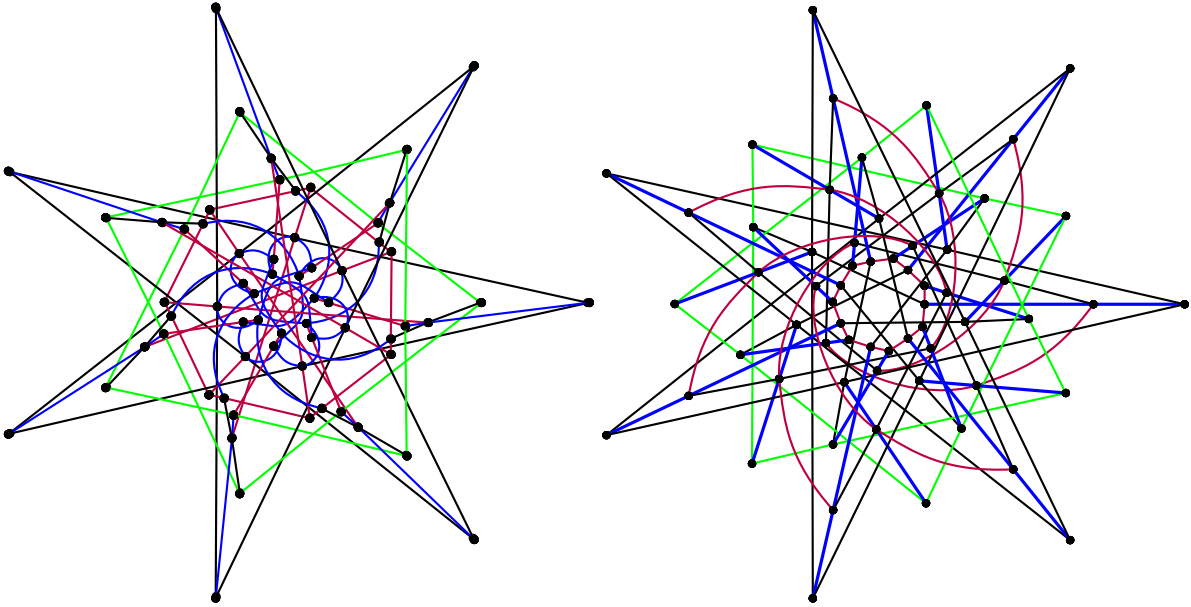
\begin{figure}[!ht]
\caption{\Small The two generalised hexagons of order $(2,2)$. 
Each is the point--line dual of the other. 
There are $(2+1)(2^4+2^2+1)=63$ points and lines, and each point (respectively line) is incident with exactly $2+1=3$ lines (respectively points). 
The Dickson group $\textup{\sf G}_2(2)$ acts primitively and distance-transitively on both points and lines. 
On the right-hand side, the blue (thick) lines form a {\em spread}, that is, every point is incident with a unique blue line. 
These pictures were inspired by a paper of Schroth~\cite{MR1675918}.}
\label{F:GH}\vskip5mm
\begin{center}
    \begin{tikzpicture}[scale=0.047]
      %\draw[step=10cm,help lines] (-90,-90) grid (90,90);
      \foreach\n in {0, 1,..., 6}{
      \begin{scope}[rotate=\n*51.4286]  % 2\pi/7 = 51.4286 degrees.
        % input 9 points of the G_2(2) generalized hexagon; see Schroth 1999
        % the GH diagrams in B. Polster, A Geometrical Picture Book
        % were created by Schroth using MetaPost.
        \coordinate (a0) at (85,0);
        \coordinate (b0) at (55,0);
        \coordinate (c0) at (12.5,0);%(13.5,0);
        \coordinate (d0) at (8.5,1.3);
        \coordinate (e0) at (16.2,9);
        \coordinate (f0) at (30,14.3);%(26.2,22.5);
        \coordinate (g0) at (26.6,17.0);%(26.6,16.0);%(25.95,14.34);
        \coordinate (h0) at (26.3,22.4);%(26.5,22.2); 
        \coordinate (i0) at (29.5,28);
        % rotate the 9 points by 1/7 turn (6 times) to get all 63 points
        \foreach\k in {1, 2,..., 6}{
          \foreach\p in {a,b,c,d,e,f,g,h,i}{
            \coordinate (\p\k) at ($ (0,0)!1! \k*51.4286:(\p0) $);
          }
        }
        \draw[thick,black] (a0)--(e1)--(a3);    % L1
        \draw[thick,green] (b0)--(h0)--(b2);    % L2
        \draw[thick,purple] (f0)--(g0)--(f1);   % L3
        \draw[thick,blue] (h0)--(i0)--(a1);     % L4
        \draw[thick,purple] (h6)--(c0)--(d0);   % L5
        \draw[thick,purple] (f0)--(e0)--(i3);   % L6
        \draw[thick,black] (b0)--(i6)--(g6);    % L7
        \draw[thick,color=blue] (g0) .. controls (27,2) and (17,-5)
           .. (d6) .. controls (-7,-5) and (-5,-5) .. (c3); % L8
        \draw[thick,color=blue] (c0) .. controls (16,3) and (17,5)
           .. (e0) .. controls (15,13) and (8,15) .. (d1);  % L9
         % draw black vertices after the lines
        \foreach\k in {0,1,..., 6}{
          \foreach\p in {a,b,c,d,e,f,g,h,i}{
            \draw[fill] (\p\k) circle [radius=1.1];
          }
        }
      \end{scope}
      }
    \end{tikzpicture}
    \begin{tikzpicture}[scale=0.4]
      %\draw[help lines] (-10,-10) grid (10,10);
      \foreach\n in {0, 1,..., 6}{
      \begin{scope}[rotate=\n*51.4286]  % 2\pi/7 = 51.4286 degrees.
        % input 9 points of the G_2(2) generalized hexagon; see Schroth 1999
        \coordinate (a0) at (10,0);
        \coordinate (b0) at (7,0); %(6.341,0);
        \coordinate (c0) at (1.45,0); %(1.342,0);
        \coordinate (d0) at (4.878,-0.4878);
        \coordinate (e0) at (2.1729,0.37976); %(1.8724,0.4761595);
        \coordinate (f0) at (1.45,0.612); %(1.342,0.64627);
        \coordinate (g0) at (2.78,-0.585);
        \coordinate (h0) at (4.074,0.7846); %(4.02439,0.7317);
        \coordinate (i0) at (6.0976,2.9268);
        % rotate the 9 points by 1/7 turn (6 times) to get all 63 points
        \foreach\k in {1, 2,..., 6}{
          \foreach\p in {a,b,c,d,e,f,g,h,i}{
            \coordinate (\p\k) at ($ (0,0)!1! \k*51.4286:(\p0) $);
          }
        }
        \draw[very thick,blue] (a0)--(b0)--(c0);    % L1 in spread
        \draw[very thick,blue] (d0)--(e0)--(f0);    % L2 in spread
        \draw[very thick,blue] (g0)--(h0)--(i0);    % L3 in spread
        \draw[thick,black] (a0)--(g1)--(a3);   % L4
        \draw[thick,green] (i0)--(d1)--(i2);   % L5
        \draw[thick,black] (c0)--(g3)--(d3);   % L6
        \draw[thick,color=purple] (b0) .. controls (5.7,-1.8) and (4.6,-2.2)
           .. (h6) .. controls (1.5,-3.2) and (-1,-2.4) .. (e4); % L7
        \draw[thick,] (b0)--(h0)--(e2);   % L8
        \draw[thick,purple] (f6)--(c0)--(f0);  % L9
        % draw black vertices after the lines
        \foreach\k in {0,1,..., 6}{
          \foreach\p in {a,b,c,d,e,f,g,h,i}{
            \draw[fill] (\p\k) circle [radius=0.12];
          }
        }
     \end{scope}
     }
    \end{tikzpicture}
    \end{center}
\end{figure}

A {\em generalised $d$-gon} is a point--line geometry whose bipartite incidence graph has diameter $d$ and girth $2d$. 
An {\em automorphism} (or {\em collineation}) of a generalised $d$-gon is a permutation of the point set, together with a permutation of the line set, such that incidence is preserved. 
To exclude trivial cases, we require that the geometry is \emph{thick}, namely that each line contains at least three points and each point lies on at least three lines. 
In this case, there are constants $s\ge 2$ and $t\ge 2$ such that each line contains exactly $s+1$ points and each point lies on exactly $t+1$ lines, and $(s,t)$ is called the {\em order} of the generalised $d$-gon (see~\cite[Corollary~1.5.3]{Van-Maldeghem:1998sj}). 
For illustration, the two generalised $6$-gons of order $(2,2)$ are shown in Figure~\ref{F:GH}. 
The celebrated theorem of Feit and Higman~\cite{Feit:1964os} shows that a thick generalised $d$-gon can only exist when $d \in \{2,3,4,6,8\}$, and as generalised $2$-gons are simply geometries whose incidence graphs are complete bipartite, they can also be regarded as trivial. 
This leaves four distinct types of thick generalised polygon. 
A generalised $3$-gon is precisely a projective plane, and it has long been conjectured (see Dembowski~\cite[p.~208]{Dembowski:1997sy}) that the mild condition of transitivity on the set of points characterises the classical projective plane $\PG(2,q)$. 
Currently the best result is to due to Gill~\cite{Gill:2007xw}, who proved that all minimal normal subgroups of a group $G$ acting transitively on a non-classical projective plane are elementary abelian. 
Moreover, Gill~\cite{Gill:ft} proved that the Sylow $2$-subgroups of $G$ are cyclic or generalised quaternion, 
and, in particular, that the only possible insoluble composition factor is $A_5$. 
Kantor~\cite{Kantor:1987qq} showed that a group acting primitively on a non-classical projective plane contains a cyclic normal subgroup of prime order acting regularly on the points. 
This leads to severe number-theoretic restrictions on the possible size of the projective plane, and Thas and Zagier~\cite{Thas:2008px} have shown that these restrictions are not satisfied for any non-classical projective plane with fewer than $4 \times 10^{22}$ points. 

Both Kantor's results, and the characterisation by Buekenhout and Van Maldeghem, rely heavily on fundamental results regarding the structure of primitive permutation groups. 
While Buekenhout and Van Maldeghem show that primitivity is a consequence of distance-transitivity, it is not necessarily the case that a flag-transitive group of automorphisms of a generalised polygon with $d \ge 4$ is primitive on points and lines: the generalised quadrangles ($4$-gons) arising from transitive hyperovals in $\PG(2,4)$ and $\PG(2,16)$ admit flag-transitive, point-primitive, but line-imprimitive automorphism groups. 
Bamberg et~al.~\cite{Bamberg:2012yf} proved that if $G$ is a group of automorphisms of a finite thick generalised quadrangle acting primitively on both points and lines, then $G$ is almost simple. 
Moreover, if $G$ is also flag-transitive then $G$ is almost simple of Lie type. 
Schneider and Van Maldeghem~\cite{Schneider:2008lr} had previously proved the following result for generalised hexagons ($6$-gons) and octagons ($8$-gons).

\begin{theorem}[Schneider and Van Maldeghem~{\cite[Theorem~2.1]{Schneider:2008lr}}]
Let $\mathcal{S}$ be a finite thick generalised hexagon or octagon. 
If $G \le \operatorname{Aut}(\mathcal{S})$ acts flag-transitively, point-primitively, and line-primitively, then $G$ is an almost simple group of Lie type. 
\end{theorem}

In this paper, we strengthen the result of Schneider and Van Maldeghem by showing that point-primitivity alone is sufficient for the same conclusion. 
That is, we prove the following theorem.

\begin{theorem} \label{thm:main}
Let $\mathcal{S}$ be a finite thick generalised hexagon or octagon. 
If $G \le \operatorname{Aut}(\mathcal{S})$ acts point-primitively, then $G$ is an almost simple group of Lie type. 
\end{theorem}

Theorem~\ref{thm:main} is proved by analysing the possible O'Nan--Scott types for $G$ considered as a primitive group on the point set of $\mathcal{S}$. 
We first prove that $G$ must be an almost simple group, that is, $G$ must have a unique minimal normal subgroup $T$, and $T$ must be a nonabelian simple group. 
It has been shown by Buekenhout and Van Maldeghem~~\cite{Buekenhout:1993bf} that $T$ cannot be a sporadic simple group, and we prove that $T$ is also not an alternating group. 
We note that our arguments do not use the Classification of Finite Simple Groups (CFSG), though the CFSG is used to conclude that $T$ must be a simple group of Lie type.

Combining Theorem~\ref{thm:main} with the previously mentioned results for projective planes and generalised quadrangles yields the following unified result for point- and line-primitive generalised polygons. 
The proof is summarised in Section~\ref{s:summary}, but we note here that the result of Bamberg et~al.~\cite{Bamberg:2012yf} for generalised quadrangles can be recast as below without the assumption of flag-transitivity. 
Recall also that the {\em socle} of a group $G$, denoted by $\operatorname{soc}(G)$, is the subgroup of $G$ generated by the minimal normal subgroups.

\begin{corollary} \label{mainCorollary}
Let $\mathcal{S}$ be a finite thick generalised $d$-gon, $d\ge 3$, of order $(s,t)$.  
If $G \le \operatorname{Aut}(\mathcal{S})$ acts point-primitively and line-primitively, then one of the following holds:
\begin{enumerate}[{\rm (i)}]
\item $G$ is almost simple of Lie type,
\item $d=3$ and $G$ is soluble, 
\item $d=4$, $G$ is almost simple with $\operatorname{soc}(G)\cong A_n$, $n\ge 5$, and $\gcd(s,t)=1$. 
\end{enumerate}
\end{corollary}

\section{Background} \label{sec2}

We first recall some basic facts about generalised polygons, referring the reader to the book by Van Maldeghem~\cite{Van-Maldeghem:1998sj} for proofs. 
Let $\mathcal{S}$ be a generalised $d$-gon, with point set $\mathcal{P}$, and line set $\mathcal{L}$. 
The requirement that the incidence graph of $\mathcal{S}$ have diameter $d$ and girth $2d$ is equivalent to the condition that 
\begin{enumerate}[(i)]
\item there are no ordinary $k$-gons in $\mathcal{S}$ for $2\le k<d$, and
\item any two elements of $\mathcal{P} \cup \mathcal{L}$ are contained in some ordinary $d$-gon.
\end{enumerate}
In particular, if $d>4$ then there are no triangles ($3$-gons) and no quadrangles in the geometry.  
If $\mathcal{S}$ is thick with order $(s,t)$ (as defined in Section~\ref{section:introduction}), then we have the following properties.
\begin{itemize}
\item[(PH)] If $\mathcal{S}$ is a generalised hexagon, then $st$ is a square, and $|\mathcal{P}| = (s+1)(s^2t^2+st+1)$.  
\item[(PO)] If $\mathcal{S}$ is a generalised octagon, then $2st$ is a square, and $|\mathcal{P}| = (s+1)(s^3t^3+s^2t^2+st+1)$.  
\end{itemize}

Given points $x$ and $y$ of a generalised hexagon or octagon, we write $x \sim y$ if $x\neq y$ and $x,y$ lie on a common line, and in this case we denote this (unique) line by $\langle x,y \rangle$. 
For several of the arguments in the proof of Theorem~\ref{thm:main}, we begin by constructing (or otherwise deducing the existence of) an automorphism $g$ such that $x\sim xg$ for some point $x$. 
The idea is then to obtain a contradiction by having $g$ fix the line $\langle x,xg \rangle$ while showing that, on the other hand, the stabilisers of $x$ and $\langle x,xg \rangle$ are equal. 
The following lemma is extremely useful for these sorts of arguments. 

\begin{lemma} \label{diamondOriginal}
Let $\mathcal{S}$ be a finite thick generalised hexagon or octagon of order $(s,t)$, and let $\mathcal{P}$ denote the set of points of $\mathcal{S}$. 
\begin{enumerate}[{\rm (i)}]
\item If $\operatorname{gcd}(s,t) \neq 1$ and $g \in \operatorname{Aut}(\mathcal{S})$ is fixed-point free, then there exists $x\in\mathcal{P}$ such that $x \sim xg$. 
\item Let $x\in\mathcal{P}$ and $g_1,g_2\in \operatorname{Aut}(\mathcal{S})$ such that $x\sim xg_1$, $x\sim xg_2$, and $xg_1g_2=xg_2g_1$. 
If $xg_1g_2 \neq x$, then $x,xg_1,xg_2$ all lie on a common line. 
\item If $H\le\operatorname{Aut}(\mathcal{S})$ is transitive on $\mathcal{P}$, then the centraliser of $H$ in $\operatorname{Aut}(\mathcal{S})$ is intransitive on $\mathcal{P}$. 
\item Let $x,y_1,y_2\in\mathcal{P}$ such that $x\sim y_1$ and $x\sim y_2$, and let $g\in\operatorname{Aut}(\mathcal{S})$ such that $xg\neq x$. 
If $g$ fixes $y_1$ and $y_2$, then $x,y_1,y_2,xg$ all lie on a common line. 
\end{enumerate}
\end{lemma}

\begin{proof}
Part~(i) is proved in \cite[Corollary~5.2 and Lemma~6.2]{Temmermans:2009zm}, and parts~(ii) and~(iii) are proved in \cite[Lemmas~3.2(iii) and~(iv)]{Schneider:2008lr}. 
It remains to prove (iii). 
First suppose, towards a contradiction, that $y_1\not\sim y_2$. 
Since $x\sim y_1$ and $x\sim y_2$ and $g$ is an automorphism (collineation), we have $xg\sim y_1g = y_1$ and $xg\sim y_2g = y_2$. 
If $xg$ lies on the line $\langle x,y_1 \rangle$, then $x,xg,y_2$ form a triangle, which is impossible. 
Similarly, $xg$ cannot lie on the line $\langle x,y_2 \rangle$. 
However, this implies that $x,y_1,xg,y_2$ form a quadrangle, which is also impossible. 
Hence $y_1\sim y_2$, and so $x,y_1,y_2$ all lie on the common line $\lambda = \langle y_1,y_2 \rangle$ because $\mathcal{S}$ contains no triangles. 
Since $g$ fixes $y_1$ and $y_2$, it fixes $\lambda$ setwise; that is, $\lambda g = \langle y_1g,y_2g \rangle = \langle y_1,y_2 \rangle = \lambda$. 
Hence $xg$ also lies on $\lambda$. 
\end{proof}

\section{The proof} \label{sec3}

We now work under the following hypothesis. 

\begin{hypothesis} \label{hyp}
Let $\mathcal{S}$ be a finite generalised hexagon or octagon, with point set $\mathcal{P}$, and order $(s,t)$, where both $s$ and $t$ are at least $2$. 
Suppose that $G \le \operatorname{Aut}(\mathcal{S})$ acts primitively on $\mathcal{P}$, and let $M$ be a minimal normal subgroup of $G$. 
\end{hypothesis}

Recall that the structure of a primitive permutation group is described by the O'Nan--Scott Theorem. 
We follow the version of the O'Nan--Scott Theorem given by Praeger~\cite[Section~5]{MR1477745}, which splits the primitive permutation groups into eight types. 
It was shown by Schneider and Van Maldeghem~\cite[Lemma~4.2(i)]{Schneider:2008lr} that if Hypothesis~$\ref{hyp}$ holds then $G$ cannot have O'Nan--Scott type HA (affine), HS (holomorph simple), or HC (holomorph compound). 
Indeed, if $G$ has one of these three types then the centraliser of $M$ in $G$ is transitive on $\mathcal{P}$, which is impossible by Lemma~\ref{diamondOriginal}(iii). 
The remaining five O'Nan--Scott types are described in Table~\ref{tab:primgroups}, and we note that in these cases $G$ has a unique minimal normal subgroup $M$, and $M$ has the form
\[
M \cong T^k, \quad \text{where $T$ is a nonabelian finite simple group and } k\ge 1. 
\] 
Here we also recall that (in general) a permutation group $G$ on a set $\Omega$ is said to be {\em semiregular} if, for every $x\in \Omega$, the stabiliser $G_x := \{ g\in G \mid xg=x \}$ of $x$ is the trivial subgroup. 
The group is said to be {\em regular} if it is semiregular and transitive, and in this case the cardinalities of $G$ and $\Omega$ are equal. 

We show in Section~\ref{s:PACD} that $G$ cannot have O'Nan--Scott type PA or CD under Hypothesis~$\ref{hyp}$, and in Section~\ref{s:SDTW} we show that $G$ cannot have type SD or TW. 
This leaves the possibility that $G$ is an almost simple group (type AS). 
An existing result of Buekenhout and Van Maldeghem~\cite{Buekenhout:1993bf} shows that an almost simple group with socle a sporadic simple group cannot act primitively (or even transitively) on the points of a finite thick generalised hexagon or octagon. 
Thus, to complete the proof of Theorem~\ref{thm:main}, it remains to show that the socle of $G$ cannot be an alternating group. 
This is done in Section~\ref{s:An}. 

\begin{center}
\small
\begin{table}
\begin{tabular}{l|p{11.2cm}} \hline 
Type & Description \\ \hline 
AS (almost simple) & $M\cong T \le G \le \operatorname{Aut}(T)$. \\ 
TW (twisted wreath) & $M\cong T^k$, $k\ge 2$, acts regularly on $\Omega$. \\ 
SD (simple diagonal) & $M\cong T^k$, $k\ge 2$, $M_{\omega}$ ($\omega\in\Omega$) is a full diagonal subgroup of $M$, $|\Omega| = |T|^{k-1}$, and $G$ acts primitively on the set of $k$ simple direct factors of $M$.\\ 
CD (compound diagonal) & $\Omega = \Gamma^\ell$ and $G\le H \text{ wr } S_\ell$, $H\le \operatorname{Sym}(\Gamma)$ primitive of type SD, $\operatorname{soc}(H) = T^{k/\ell}$, $k\ge 2$ and $k/\ell\ge 2$; $G$ acts transitively on the simple direct factors of $M \cong T^k$. \\ 
PA (product action) & $\Omega = \Gamma^k$ and $G\le H \text{ wr } S_k$, $H\le \operatorname{Sym}(\Gamma)$ primitive of type AS, $\operatorname{soc}(H) \cong T$; $G$ acts transitively on the simple direct factors of $M \cong T^k$, $k\ge 2$. \\ \hline
\end{tabular} 
\vspace{2mm}
\caption{\small Five of the possible O'Nan--Scott types of primitive groups $G \le \operatorname{Sym}(\Omega)$. 
Here $M$ is the unique minimal normal subgroup of $G$, and $T$ denotes a nonabelian finite simple group. 
}
\label{tab:primgroups}
\end{table}
\end{center}

\subsection{Types PA and CD} \label{s:PACD}

Suppose that $G$ has O'Nan--Scott type PA or CD under Hypothesis~$\ref{hyp}$. 
Then $\mathcal{P}$ can be identified with a Cartesian product $\Gamma^\ell$ in such a way that $G$ embeds in the wreath product $H \operatorname{wr} S_\ell$, where $H$ is a primitive subgroup of $\operatorname{Sym}(\Gamma)$ and $G$ induces a transitive subgroup of $S_\ell$. 
Write $N = \operatorname{soc}(H)$ and note that $N^\ell = \operatorname{soc}(G) = M$. 
Let $\alpha$ be an arbitrary element of $\Gamma$, and consider the point $x\in \Gamma^\ell$ represented by the $\ell$-tuple $(\alpha,\ldots,\alpha)$. 
Our argument is in three steps, marked (i)--(iii) below. 
In step, (i) we show that there exists an element $y$ collinear with $x$ such that $y$ is represented by the $\ell$-tuple $(\beta,\alpha,\ldots,\alpha)$ for some $\beta\in \Gamma \backslash \{ \alpha \}$.
In step, (ii) we show that $G_\lambda = G_x$, where $\lambda$ is the line collinear with both $x$ and $y$. 
Finally, in step (iii) we construct an automorphism $g$ that fixes $\lambda$ but not $x$, thereby obtaining a contradiction and hence proving that $G$ in fact cannot have type PA or CD.

\begin{enumerate}[(i)]
\item {\em There exists a point $y$ such that $y\sim x$ and $y=(\beta,\alpha,\ldots,\alpha)$ for some $\beta\in \Gamma \backslash \{ \alpha \}$.} 
This is established by Schneider and Van Maldgehem~\cite[proof of Lemma~4.2(ii)]{Schneider:2008lr} without using flag-transitivity, but we include a proof to make it clear that flag-transitivity is not needed. 
Let $y=(\beta_1,\beta_2,\ldots,\beta_\ell)$ be a point collinear with $x$, and suppose that $y$ has $i$ entries different from $\alpha$. 
If $i=1$ then, without loss of generality, $y$ differs from $x$ in the first component and we are done. 
Now assume that $i\ge 2$. 
Then, without loss of generality, $\beta_1 \neq \alpha \neq \beta_2$ and, if $i<\ell$, $\beta_{i+1} = \cdots = \beta_\ell = \alpha$. 
Since $G$ has type PA or CD, the primitive group $H$ has type AS or SD (respectively), so $N=\operatorname{soc}(H)$ is not regular and it follows from \cite[Corollary~2.2(a)]{Betten:2003vl} that the only point of $\Gamma$ fixed by $N_\alpha$ is $\alpha$. 
In particular, $N_\alpha \neq N_{\beta_1}$, and hence there exists $g\in N_\alpha$ such that $\beta_1' := \beta_1 g \neq \beta_1$. 
Set $\bar{g} = (g\; 1\ldots 1)$ and $y' = y\bar{g} = (\beta_1',\beta_2,\ldots,\beta_\ell)$. 
Then $\bar{g} \in M \le G$, and $\bar{g}$ fixes $x$ so $x\sim y'$. 
Similarly, we can choose $h \in N_{\beta_2}$ such that $\alpha' := \alpha h \neq \alpha$, and we set $\bar{h} = (1\; h\; 1\ldots 1) \in M$. 
Then $x \neq x\bar{h} = (\alpha, \alpha', \alpha, \ldots, \alpha)$. 
That is, $x\bar{h}$ differs from $x$ in only one component. 
To complete the proof of (i), it suffices to check that $x \sim x\bar{h}$. 
Since $\bar{h}$ fixes both $y$ and $y'$ but not $x$, Lemma~\ref{diamondOriginal}(iv) implies that $x,x\bar{h},y,y'$ all lie on a common line. 
In particular, $x\sim x\bar{h}$.

\item {\em $G_\lambda=G_x$, where $\lambda := \langle x,y \rangle$ with $y$ given by \textnormal{(i)}.} 
The following argument is also adapted from \cite[proof of Lemma~4.2(ii)]{Schneider:2008lr}. 
We show that $G_x \le G_\lambda$, which implies that $G_\lambda = G_x$ because $G_x$ is a maximal subgroup of the primitive group $G$ and because $G_\lambda \neq G$ (since $G$ acts transitively on the set of all points of $\mathcal{S}$, it cannot stabilise a single line). 
Let $g\in G_x$ and write $g = (g_1,\ldots,g_\ell) \sigma \in G$, where $g_1,\ldots,g_\ell\in H$ and $\sigma \in S_\ell$, so that $\alpha g_i = \alpha$ for all $i \in \{ 1,\ldots,\ell \}$. 
If $g$ fixes $y$ then it fixes $\lambda$ setwise as required, so assume that $yg \neq y$. 
We show that $yg$ lies on $\lambda$, which also implies that $g$ fixes $\lambda$ because then $\lambda g = \langle x,y \rangle g = \langle xg,yg \rangle = \langle x,yg \rangle = \lambda$. 
First suppose that $1\sigma \neq 1$. 
Then all components of $yg$ are equal to $\alpha$, except the component in position $1\sigma$, which is equal to $\beta g_1$. 
Choose $h, h'\in N$ such that $\alpha h = \beta$ and $\alpha h' = \beta g_1$, and set $\bar{h} = (h\; 1\; \ldots\; 1)$ and $\bar{h}' = (1\; \ldots\; 1\; h'\; 1\; \ldots\;1)$, where $h'$ appears in position $1\sigma$ of $\bar{h}'$. 
Then $x \sim x\bar{h} = y$, $x \sim x\bar{h}' = yg$, and $x\bar{h}\bar{h}' = x\bar{h}'\bar{h} \neq x$, so Lemma~\ref{diamondOriginal}(ii) implies that $x,y,yg$ all lie on a common line. 
That is, $yg$ lies on $\lambda$. 
Now suppose that $1\sigma = 1$. 
In this case, $yg = (\beta g_1,\alpha,\ldots,\alpha)$. 
Take any $g' = (g_1',\ldots,g_\ell')\sigma' \in G_x$ such that $1\sigma' \neq 1$. 
We have just shown that $g'$ fixes $\lambda$, so in particular $yg'$ lies on $\lambda$. 
Without loss of generality, we may assume that $1\sigma' = 2$, so that $yg' = (\alpha,\beta g_1',\alpha,\ldots,\alpha)$. 
Now let $\bar{h}'' = (h'\; 1\; \ldots\; 1)$, where $\alpha h' = \beta g_1$ as above, and let $\bar{h}''' = (1\; h''\; 1\; \ldots\; 1)$, where $\alpha h'' = \beta g_1'$. 
Then $x\bar{h}'' = yg$, $x\bar{h}''' = yg'$, and $x\bar{h}''\bar{h}''' = x\bar{h}'''\bar{h}'' \neq x$, so Lemma~\ref{diamondOriginal}(ii) implies that $yg,yg',x$ all lie on a common line, and hence $yg$ lies on $\lambda$ as required.

\item {\em A contradiction: there exists $g \in G_\lambda$ such that $g \not \in G_x$.} 
Choose $g \in G_y$ such that $g=(g_1,\ldots,g_\ell) \sigma$ with $1\sigma = 2$. 
Such an element exists because the stabiliser of any point in $G$ is transitive on the simple direct factors of $M$. 
Let $i$ satisfy $i \sigma =1$, and note that $i \neq 1$. 
Then $(\beta,\alpha,\ldots,\alpha) = y = yg = (\beta,\alpha,\ldots,\alpha)g = (\alpha g_i,\beta g_1,\ldots)$, where the components from the third position onwards are of the form $\alpha g_j$ with $j \not \in \{1,i\}$. 
That is, $\alpha g_i = \beta$, $\beta g_1=\alpha$, and $\alpha g_j = \alpha$ for all $j \not \in \{1,i\}$.
Next, observe that $xg \neq x$, $xg \neq y$, and $xg \sim y$. 
Indeed, $xg = (\alpha g_i,\alpha g_1,\alpha,\ldots,\alpha) = (\beta,\alpha g_1,\alpha,\ldots,\alpha)$ is not equal to $x$ or $y$ because $\beta g_1=\alpha$ and hence $\alpha g_1 \neq \alpha$, and $x\sim y$ implies $xg\sim yg=y$. 
Choose $h', h'' \in N = \operatorname{soc}(H)$ such that $\beta h' = \alpha$ and $\alpha h'' = \alpha g_1$, and write $\bar{h}' = (h'\; 1\ldots 1)$ and $\bar{h}'' = (1\; h''\; 1\ldots 1)$. 
Then $y\bar{h}'=x$, $y\bar{h}''=xg$, $\bar{h}'$ and $\bar{h}''$ commute, and $y\bar{h}'\bar{h}'' \neq y$, so Lemma~\ref{diamondOriginal}(ii) implies that $xg$ lies on $\lambda$. 
It follows that $\lambda g = \langle x,y \rangle g = \langle xg,yg \rangle = \langle xg,y \rangle = \lambda$, namely that $g \in G_\lambda$. 
However, $G_\lambda = G_x$ from (ii), so we have a contradiction because $g \not \in G_x$. 
\end{enumerate} 
To summarise, we have proved the following result.

\begin{lemma} \label{lemma:PACD}
If Hypothesis~$\ref{hyp}$ holds then the O'Nan--Scott type of $G$ is not PA or CD.
\end{lemma}

\subsection{Types SD and TW} \label{s:SDTW}

We begin with two lemmas, from which it is then deduced that $G$ cannot have O'Nan--Scott type SD or TW under Hypothesis~$\ref{hyp}$. 
For the second lemma, recall again that a permutation group is said to be {\em semiregular} if every point stabiliser is trivial. 

\begin{lemma} \label{stNotCoprime}
Let $s$ and $t$ be positive integers such that $\gcd(s,t)=1$. 
\begin{enumerate}[{\rm (i)}]
\item If $st$ is a square then $(1+s)(1+st+s^2t^2)$ is not divisible by $4$. 
\item If $2st$ is a square then $(1+s)(1+st)(1+s^2t^2)$ is not divisible by $4$. 
\end{enumerate} 
\end{lemma}

\begin{proof}
(i) If $4$ divides $(1+s)(1+st+s^2t^2)$ then $s \equiv 3\pmod 4$ because $(1+st+s^2t^2)$ is odd. 
However, $st$ is a square and $\gcd(s,t) = 1$, so $s$ must be a square and hence $s \not \equiv 3\pmod 4$. 

(ii) If $4$ divides $(1+s)(1+st)(1+s^2t^2)$ then $s \equiv 3\pmod 4$ because $(1+st)(1+s^2t^2)$ is odd, since $st$ is even if $2st$ is a square. 
In particular, $s$ is odd. 
However, if $s$ odd is then $s$ must be a square because $2st$ is a square and $\operatorname{gcd}(s,t) = 1$, so $s \not \equiv 3\pmod 4$. 
\end{proof}

\begin{lemma} \label{semiregular}
Suppose that Hypothesis~$\ref{hyp}$ holds and that $M \cong T_1\times\cdots\times T_k$ for some pairwise isomorphic nonabelian finite simple groups $T_1,\ldots,T_k$, with $k\ge 2$. 
Then, if $\gcd(s,t)\ne 1$, there exist distinct $i,j \in \{1\ldots,k\}$ such that $T_i\times T_j$ is not semiregular on $\mathcal{P}$.
\end{lemma}

\begin{proof}
Suppose that $\gcd(s,t)\ne 1$ but that $T_i\times T_j$ is semiregular for all distinct $i,j$. 
Then, in particular, $T_1$ is semiregular. 
Choose an involution $h\in T_1$, namely an element of order $2$. 
Such an element exists because the nonabelian finite simple group $T_1$ has even order, by the Feit--Thompson Theorem~\cite{MR0166261}. 
Since $h$ does not fix any point, Lemma~\ref{diamondOriginal}(i) says that there exists $x\in\mathcal{P}$ such that $x\sim xh$, and so $h$ fixes the line $\lambda = \langle x,xh \rangle$ setwise. 
We now show that $G_\lambda = G_x$, which is a contradiction because $h$ does not fix $x$. 
We first claim that $\lambda$ is fixed (setwise) by every element of $G_x$ that does not normalise $T_1$. 
Let $g\in G_x$ be such an element. 
Since $x=xg$ and $x \sim xh$, we have $x = xg \sim xhg = xgg^{-1}hg = xh^g$. 
Since $g$ does not normalise $T_1$, there exists $i\neq 1$ such that $T_1^g = T_i$. 
Thus $h^g$ commutes with $h$, and moreover, $hh^g$ lies in the semiregular group $T_1 \times T_i$ and hence does not fix $x$. 
Lemma~\ref{diamondOriginal}(ii) therefore implies that $x,xh,xh^g$ all lie on a common line. 
Therefore, $xh^g$ lies on $\lambda$, and hence $\lambda g= \langle x,xh \rangle g= \langle xg,xhg \rangle = \langle x,xh^g \rangle =\lambda$. 
That is, $g\in G_\lambda$ as claimed. 
Now take $a \in \mathbb{N}_{G_x}(T_1)$ and $b \in G_x \backslash \mathbb{N}_{G_x}(T_1)$. 
By the claim, both $ab$ and $b$ belong to $G_\lambda$, since $ab \not \in \mathbb{N}_{G_x}(T_1)$, and hence $a = (ab)b^{-1}\in G_\lambda$. 
Thus $G_x \le G_\lambda$, and since $G_x$ is a maximal subgroup of $G$ and $G_\lambda \neq G$ (as noted in Section~\ref{s:PACD}), it follows that $G_x = G_\lambda$ as required. 
\end{proof}

\begin{lemma} \label{lemma:SDTW}
If Hypothesis~$\ref{hyp}$ holds then the O'Nan--Scott type of $G$ is not SD or TW.
\end{lemma}

\begin{proof}
Write $M = T^k$, where $T$ is a nonabelian finite simple group (as in Table~\ref{tab:primgroups}). 
If $G$ has type TW then $M$ acts regularly on $\mathcal{P}$, and if $G$ has type SD then $M = R \times T$, where $R := T^{k-1}$ acts regularly on $\mathcal{P}$. 
In either case, the cardinality of $\mathcal{P}$ is divisible by $4$, because the order of every nonabelian finite simple group is divisible by $4$ (this is a well-known consequence of the Feit--Thompson Theorem, as explained in~\cite[pp.~1443--1444]{Schneider:2008lr}). 
It therefore follows from Lemma~\ref{stNotCoprime} and properties (PH) and (PO) in Section~\ref{sec2} that $\gcd(s,t)>1$, and Lemma~\ref{semiregular} then contradicts the regularity of $M$ in the TW case, and of $R$ in the SD case provided that $k\ge 3$. 
It remains to consider the case where $G$ has type SD with $k=2$. 
Here, $T_1$ and $T_2$ are both transitive minimal normal subgroups of $M$, and they centralise each other. 
This contradicts Lemma~\ref{diamondOriginal}(iii), and hence this case also cannot occur. 
\end{proof}

\begin{remark} \label{rem:Parkinson}
We note that a more general version of Lemma~\ref{diamondOriginal}(i) is given by Parkinson et~al. \cite[Theorem~10.2]{Parkinson}. 
It says that, without the assumption $\gcd(s,t)>1$, every fixed-point free automorphism will either map some point $x$ to a point collinear with $x$, or it will map some point $x$ to a point at distance $4$  from $x$ in the incidence graph of $\mathcal{S}$. 
This implies the conclusion of Lemma~\ref{semiregular} without having to assume that $\gcd(s,t)>1$: if the point $xh$ in the proof is instead at distance $4$ from $x$, then there is a point $y$ with $x\sim y\sim xh$, but the fixed-point free involution $h$ swaps $x$ and $xh$ without fixing $y$, so a similar argument to that in the proof of Lemma~\ref{diamondOriginal}(iv) implies that $x\sim xh$, a contradiction.
On the other hand, we actually know that $\gcd(s,t)>1$ in the situation considered above (via Lemma~\ref{stNotCoprime}), so in a sense it is more natural to argue as we have done. 
(It is an open question whether the order $(s,t)$ of a finite thick generalised hexagon or octagon always satisfies $\gcd(s,t)>1$.) 
\end{remark}

\subsection{Type AS with socle an alternating group} \label{s:An}

By Lemmas~\ref{lemma:PACD} and~\ref{lemma:SDTW}, if Hypothesis~\ref{hyp} holds then $G$ must be an almost simple group. 
We now treat the case where $M=T$ is an alternating group. 
That is, we prove the following result.

\begin{lemma} \label{lemma:An}
If Hypothesis~$\ref{hyp}$ holds with $G$ an almost simple group, then the socle of $G$ is not an alternating group. 
\end{lemma}

For the proof, suppose towards a contradiction that Hypothesis~\ref{hyp} holds with $G$ almost simple and $\operatorname{soc}(G) = A_n$ for some $n\ge 5$.  
Buekenhout and Van Maldeghem~\cite{Buekenhout:1993bf} have shown that such a group $G$ cannot act transitively on the points of a generalised hexagon or octagon if $n<14$, so we may assume that $n \ge 14$. 
In particular, we have $G=A_n$ or $S_n$ because $n\neq 6$. 
Our analysis splits into three cases, depending on whether the stabiliser $G_x$ of a point $x\in \mathcal{P}$ is an intransitive, transitive but imprimitive, or primitive subgroup of $S_n$ in the natural action on $\{1,\ldots,n\}$.

\subsubsection{Intransitive point stabiliser} \label{ss:intrans}

Let $x\in \mathcal{P}$ and suppose that $G_x$ acts intransitively on $\{1,\ldots,n\}$. 
Then $G_x$ stabilises a partition of $\{1,\ldots,n\}$ into two blocks, one of size $k$, say, and one of size $\ell$, where $k+\ell=n$. 
If $k=\ell$ then $G_x < G\cap (S_k \text{ wr } S_2) < G$, so $G_x$ is not a maximal subgroup, contradicting the primitivity of $G$ on $\mathcal{P}$. 
We may therefore assume, without loss of generality, that $k<\ell$. 
We then have $(A_k \times A_\ell) \cdot 2 \le G_x \le S_k \times S_\ell$, and the points of $\mathcal{S}$ can be labelled by $k$-element subsets of $\{ 1,\ldots,n \}$. 
We note also that $G_x$ must have at least four orbits on $\mathcal{P}$, as it preserves distance in the incidence graph of $\mathcal{S}$ and there are points in $\mathcal{P}$ at distances $0$, $2$, $4$, and $6$ (and even $8$ in the case of a generalised octagon) from $x$. 
If $k<3$ then the number of orbits of $G_x$ is less than four, and hence we may assume that $k\ge 3$. 

The following facts are proved by Bamberg et~al.~\cite[Lemmas~5.5 and~5.6]{Bamberg:2012yf} in the case of a generalised quadrangle (with the same assumptions as above). 
We note that the proofs given there are also valid for generalised hexagons and octagons, but we include proofs to make this clear.
\begin{itemize}%[(i)]
\item[(F1)] For every $i \in \{ 1,\ldots,k \}$, if $x, y \in \mathcal{P}$ are collinear and $|x\cap y| = i$, then any $x',y' \in \mathcal{P}$ with $|x'\cap y'| = i$ are also collinear. 
\item[(F2)] For every $i \in \{ 1,\ldots,k \}$, if $x, y \in \mathcal{P}$ are collinear and $|x\cap y| = i$, then there exists $y' \in \mathcal{P}$ such that $|x\cap y'|=i$ and $y' \not \sim y$.
\end{itemize}

\begin{proof}[Proof of~(F1) and~(F2)]
For (F1), it suffices to observe that $G=A_n$ or $S_n$ preserves collinearity and is transitive on pairs of $k$-subsets of $\{1,\ldots,n\}$ with intersection size $i$. 
For (F2), suppose towards a contradiction that every point $y'$ with $|x\cap y'|=i$ is collinear with $y$. 
By (F1), every such point $y'$ is also collinear with $x$, and hence lies on the line $\lambda := \langle x,y \rangle$ (because $\mathcal{S}$ contains no triangles). 
Let $J$ denote the generalised Johnson graph with vertices the $k$-subsets of $\{ 1,\ldots,n \}$ and two vertices adjacent if and only if their labels intersect in $i$ elements. 
If $n\neq k/2$ then $G$ acts primitively on the point set of $J$, and a partition into connected components is $G$-invariant, so $J$ is a connected graph. 
We prove by induction on the distance $\delta(x,x_1)$ between $x$ and $x_1\in J$ that all vertices of $J$ lie on the line $\lambda$. 
The inductive hypothesis is true for distance $0$ because $x$ lies on $\lambda$, and for distance $1$ by our assumption that every point $y'$ with $|x\cap y'|=i$ is collinear with $y$. 
Assume that it is true for distance $d$, and suppose that $\delta(x,x_1)=d+1$. 
Then $x_1$ has a neighbour $x_2\in J$ with $\delta(x,x_2)=d$, and thus, by the inductive hypothesis, $x_2$ lies on $\lambda$. 
Also, $x_2$ has a neighbour $x_3\in J$ such that $\delta(x,x_3)=d-1$, and hence $x_3$ lies on $\lambda$. 
If $x_1$ did not lie on $\lambda$ then $x_2$ would have two neighbours, $x_1$ and $x_3$, such that $|x_2\cap x_1| = |x_2\cap x_3| = i$ but $x_1$ and $x_3$ are not on the same line, so by vertex-transitivity the same would have to be true for $x$, a contradiction. 
Hence all vertices of $J$ lie on $\lambda$. 
However, this is a contradiction because the points of $\mathcal{S}$ do not all lie on a single line.  
\end{proof}

Now let $x$ denote the point with label $\{ 1,\ldots,k \}$, and let $k_1<k$ be maximal such that there exists a point $y \sim x$ with $|x\cap y| = k_1$. 
We claim that $k_1=0$. 
We first show that $k_1<k-1$ by adapting an argument from \cite[Section~5]{Bamberg:2012yf}. 
Suppose, towards a contradiction, that $k_1=k-1$. 
Then, without loss of generality, $y$ has label $\{1,\ldots,k-1,k+1 \}$. 
By (F2), there exists $y'\in \mathcal{P}$ such that $|x\cap y'|=k-1$ and $y' \not \sim y$, and by (F1) we have $y' \sim x$ and $|y\cap y'| \neq k-1$. 
In particular, without loss of generality we can write $y' = \{2,\ldots,k,k+2\}$. 
However, the automorphism $(1\; k+1)(k\; k+2)$ fixes both $y$ and $y'$ but does not fix $x$, so Lemma~\ref{diamondOriginal}(iv) implies that $y\sim y'$, a contradiction. 
Hence $k_1<k-1$. 
In particular, if $k=2$ then $k_1=0$, so we can now assume that $k\ge 3$ and complete the proof of the claim that $k_1=0$ by adapting an argument from \cite[Section~5]{Schneider:2008lr}. 
Without loss of generality, $y = \{ 1,\ldots,k_1,k+1,\ldots,2k-k_1 \}$. 
Let $z = \{ 1,\ldots,k_1,k+2,\ldots,2k-k_1+1 \}$. 
Then $y\not \sim z$ because $|y\cap z| = k-1>k_1$. 
On the other hand, $|x\cap z| = k_1$, and hence $x\sim z$ by (F1). 
If $k_1>0$ then the automorphism $(1\; k+2)(k-1\; k)$ fixes $y$ and $z$ but not $x$, so Lemma~\ref{diamondOriginal}(iv) implies that $y\sim z$, a contradiction. 
Therefore, $k_1=0$ as claimed. 
However, now if $2k+1<n$ then the automorphism $g=(1\; 2k+2)(2\; 3)$ fixes $y = \{k+1,\ldots,2k\}$ and $z = \{k+2,\ldots,2k+1\}$ but not $x$, a contradiction according to Lemma~\ref{diamondOriginal}(iv). 
Therefore, $n=2k+1$. 
However, now by maximality of $k_1=0$ and transitivity of $A_k$, there are precisely $k+1$ points collinear with $x$, and $G_x$ acts $2$-transitively on this set of $k+1$ points. 
This implies that there is either only one line incident with $x$, or $k+1$ such lines, each incident with only two points. 
Either situation contradicts the thickness of $\mathcal{S}$.

\subsubsection{Transitive but imprimitive point stabiliser} \label{ss:imprim}

In this case, $G_x$ is the stabiliser of a partition of $\{ 1,\ldots,n \}$ into $\ell$ blocks of size $k$, where $n=k\ell$. 
First suppose that $\ell=2$, and let $H$ denote the stabiliser of the point $n$ in the natural action of $G$ on $\{1,\ldots,n\}$. 
Since $G_x$ is transitive on $\{1,\ldots,n\}$, we have $G=HG_x$, and this in turn implies that $H$ is transitive on $\mathcal{P}$. 
Moreover, $H_x = H\cap G_x = (S_k\times S_{k-1})\cap G$, which is a maximal subgroup of $H$. 
Therefore, $H$ is primitive on $\mathcal{P}$. 
However, by Section~\ref{ss:intrans}, $\mathcal{S}$ cannot admit a point-primitive action of $H$ with stabiliser intransitive on $\{1,\ldots,n-1\}$. 
Therefore, $\ell\ge 3$.

Now, given points $x,y\in \mathcal{P}$, we use the obvious notation $|x\cap y|$ to mean the number of partition classes common to $x$ and $y$. 
We consider, in particular, points for which $|x\cap y| = \ell-2$, and for this case we define the following additional notation. 

\begin{definition} \label{l-2}
Let $x,y\in \mathcal{P}$ such that $|x\cap y| = \ell-2$. 
Then there are exactly two partition classes $B_1,B_2$ of $x$ that do not belong to $y$, and exactly two partition classes $B_1',B_2'$ of $y$ that do not belong to $x$. 
By appropriate labelling, we may assume that $|B_1\cap B_1'| = |B_2\cap B_2'| \ge \lceil k/2 \rceil$, and we then write $|x-y|_{\ell-2} := |B_1\cap B_1'| = |B_2\cap B_2'|$. 
\end{definition}

We first claim that there exist collinear points $x,y\in\mathcal{P}$ such that $|x\cap y| = \ell-2$ and $|x-y|_{\ell-2} = k-1$. 
The proof of this claim is via steps (i)--(iii) below, with steps (i) and (ii) adapted from Schneider and Van Maldeghem~\cite[p.~1447]{Schneider:2008lr} but repeated here to make it clear that flag-transitivity is not required. 

\begin{enumerate}[(i)]
\item {\em There exist collinear points $x,y\in\mathcal{P}$ with $|x\cap y|\ge 1$.} 
To prove this, first choose a point $x \in \mathcal{P}$ and suppose, without loss of generality, that $x$ is labelled by 
\[
x = \{ B_1,\ldots,B_\ell \}, \quad \text{where } B_i = \{(i-1)k+1,(i-1)k+2,\ldots,ik\} \text{ for } i=1,\ldots,\ell-1.
\]
Choose $y\in \mathcal{P}$ such that $x\sim y$. 
If $|x\cap y| \ge 1$ then we are done, so suppose that $|x\cap y|= 0$. 
First consider the case $k=2$. 
Then $\ell\ge 7$ because $n=k\ell\ge14$. 
Observe that the automorphism $g=(1\;2)(3\;4)$ fixes $x$ and fixes at least $\ell-\operatorname{supp}(g) = \ell-4 \ge 3$ partition classes of $y$ (here $\operatorname{supp}(g)$ is the {\em support} of $g$, namely the subset of $\{1,\ldots,n\}$ of elements moved by $g$).
In particular, $|y\cap yg| \ge 3$, and we show that $y\sim yg$. 
Let $\{i_1,i_2\}$, $\{i_3,i_4\}$ be two (partition) classes common to $y$ and $yg$. 
Since $\ell-\operatorname{supp}(g) \ge 3$, we may assume that $\{i_1,i_2,i_3,i_4\}$ is not a union of two classes of $x$. 
Then the automorphism $g' = (i_1\;i_2)(i_3\;i_4)$ fixes both $y$ and $yg$, but does not fix $x$ (because $i_1,i_2$ lie in the same class of $y$ and hence in different classes of $x$, and these two classes of $x$ are not fixed setwise by $g'$), so Lemma~\ref{diamondOriginal}(iv) implies that $y\sim yg$ as required. 
Now suppose that $k\ge 3$. 
The automorphism $h = (1\;2\;3)$ fixes $x$ and fixes at least $\ell-3$ classes of $y$. 
Hence, if $\ell\ge 4$, then $|y\cap yh| \ge 1$. 
If $\ell=3$ then $k\ge 5$ since $n\ge 14$, so, in particular, there is a class of $y$ that shares at least two elements with some class of $x$. 
Without loss of generality, we may assume that $1,2$ lie together in a class of $y$, and hence that $h$ fixes at least one class of $y$. 
Thus we have $|y\cap yh| \ge 1$ for all $\ell\ge 3$, and we check that $y\sim yh$. 
Let $\{i_1,\ldots,i_k\}$ be a class common to $y$ and $yh$, and suppose without loss of generality that $i_1,i_2,i_3$ do not all lie together in a class of $x$. 
Then the automorphism $(i_1\;i_2\;i_3)$ fixes $y$ and $yh$ but does not fix $x$, so Lemma~\ref{diamondOriginal}(iv) implies that $y\sim yh$. 

\item {\em There exist collinear points $x,y\in\mathcal{P}$ with $|x\cap y|=\ell-2$.} 
This follows immediately from part~(i) if $\ell=3$, so suppose that $\ell\ge 4$. 
Let $x,y \in \mathcal{P}$ be collinear points with $|x\cap y|\ge 1$, and label $x$ as in part~(i). 
If $|x\cap y| = \ell-2$ then we are done, so suppose the contrary. 
Suppose, without loss of generality, that $y$ contains the class $B_1 = \{1,\ldots,k\}$. 
Choose a class $B_i$ of $x$ that does not belong to $y$, and choose $j_1,j_2 \in B_i$ such that $j_1,j_2$ do not lie in the same class of $y$. 
Then the automorphism $(1\;2)(j_1\;j_2)$ fixes $x$ and maps $y$ to a point $z$ such that $|y\cap z| = \ell-2$, and we check that $y\sim z$. 
Since $\ell \ge 4$, $y$ and $z$ contain a common class $B'$ that is not a class of $x$ (because otherwise $|x\cap y|=\ell-2$, contrary to our assumption). 
Hence, taking $j_3,j_4 \in B'$ lying in different classes of $x$, the automorphism $(1\;2)(j_3\;j_4)$ fixes $y$ and $z$ but does not fix $x$, so Lemma~\ref{diamondOriginal}(iv) implies that $y\sim z$. 

\item {\em With notation as in Definition~$\ref{l-2}$, there exist collinear points $x,y\in\mathcal{P}$ with $|x\cap y| = \ell-2$ and $|x-y|_{\ell-2}=k-1$.} 
This follows immediately from part~(ii) if $k \in \{2,3\}$, so suppose that $k\ge 4$. 
Let $x,y\in\mathcal{P}$ be collinear points with $|x\cap y| = \ell-2$, and again label $x$ as in part~(i). 
Assume without loss of generality that $y$ also contains the classes $B_3,\ldots,B_\ell$. 
Let $B_1',B_2'$ denote the remaining two classes of $y$, and label these classes such that $k_1 := |B_1\cap B_1'| = |B_2\cap B_2'| \ge \lceil k/2 \rceil$. 
That is, $k_1 = |x-y|_{\ell-2}$. 
Consider first the case $k\ge 5$. 
If $k_1 = k-1$ then we are done, so suppose the contrary. 
Then $\lceil k/2 \rceil \le k_1 \le k-2$, and since $k\ge 5$ we have $k_1 \ge 3$. 
Without loss of generality, $B_1' = \{ 1,\ldots,k_1,k+1,\ldots,2k-k_1\}$ and $B_2' = \{ k_1+1,\ldots,k,2k-k_1+1,\ldots,2k\}$. 
Now consider the automorphism $g=(1\;2)(k_1\;k)$. 
Then $g$ fixes $x$, and the point $yg$ contains the classes $B_3,\ldots,B_\ell$ and its other two classes are $B_1''= \{ 1,\ldots,k_1-1,k,\ldots,2k-k_1 \}$ and $B_2''=\{ k_1,\ldots,k-1,2k-k_1+1,\ldots,2k\}$. 
In particular, $|y\cap yg| = \ell-2$ and $|y-yg|_{\ell-2} = |B_1''\cap B_1'| = |B_2''\cap B_2'| = k-1$. 
We now show that $y\sim yg$. 
Since $x\sim y$, we have $x=xg\sim yg$. 
Consider the automorphism $h =(1\;2)(k_1-1\;k+2)$. 
Since the points $1,2,k_1-1,k+2$ all lie in $B_1'\cap B_1''$, $h$ fixes both $y$ and $yg$ (note that this holds even if $k_1=3$). 
However, $h$ does not fix $x$, so Lemma~\ref{diamondOriginal}(iv) implies that $y\sim yg$. 
Finally, if $k=4$ then $2 = \lceil k/2 \rceil \le k_1 \le k-2 = 2$, so $k_1=2$ and without loss of generality we have $B_1' = \{1,2,5,6\}$ and $B_2' = \{3,4,7,8\}$. 
Then $|yg'-y|_{\ell-2}=3$, where $g' := (2\;4)(5\;6)$, and $h' := (1\;5)(3\;7)$ fixes $y$ and $yg'$ but not $x$, so Lemma~\ref{diamondOriginal}(iv) implies that $y\sim yg'$. 
\end{enumerate}

Let us now fix collinear points $x$ and $y$ such that $|x\cap y|=\ell-2$ and $|x-y|_{\ell-2}=k-1$, and set up some further notation. 
Specifically, we assume that $x,y\in \mathcal{P}$ are collinear points labelled by
\begin{align*}
x &= \{ B_1, B_2, B_3,\ldots,B_\ell \}, \\
y &= \{ \{1,\ldots,k-1,k+1\}, \{k,k+2,\ldots,2k\}, B_3,\ldots,B_\ell \},
\end{align*}
where $B_i = \{(i-1)k+1,(i-1)k+2,\ldots,ik\}$ for $i=1,\ldots,\ell-1$ as before. 
We let $\lambda$ denote the line incident with both $x$ and $y$. 
Further, we define an automorphism $h$ and a subset $\Delta$ of $\{1,\ldots,n\}$ as follows:
\[
h := \begin{cases}
(1\;2)(k\;k+1) & \text{if } k\ge 3 \\
(1\;3\;2) & \text{if } k=2
\end{cases}, \quad
\Delta := \begin{cases}
\{1,2,k,k+1\} & \text{if } k\ge 3 \\
\{1,2,3\} & \text{if } k=2.
\end{cases}
\]

First observe that $y=xh$. 
In particular, $x\sim xh$. 
If $k\ge 3$ then $h$ has order $2$, and hence $h$ fixes $\lambda$. 
If $k=2$ then $h$ has order $3$ and we have $xh\sim xh^2$ and hence $xh^2\sim xh^3 = x$, so again $h$ fixes $\lambda$ because $\mathcal{S}$ contains no triangles. 
In either case, we have $h \in G_\lambda \backslash G_x$. 
To complete the proof that $G_x$ cannot be transitive but imprimitive in its natural action on $\{1,\ldots,n\}$, we now obtain a contradiction by showing that $G_x = G_\lambda$. 
The cases (i) $k=2$ and (ii) $k\ge 3$ are treated separately. 

\begin{enumerate}[(i)]
\item Suppose that $k=2$, and recall that in this case $\ell\ge 7$, because $n\ge 14$. 
In particular, we do not need to consider the case $\ell=3$, for which the following argument does not work. 
We first claim that if $g\in G_x$ is an element satisfying $\Delta g \cap \Delta = \varnothing$ and $yg \neq yh$, then $g\in G_\lambda$. 
To prove this, first note that $h=(1\;3\;2)$ commutes with $h^g$ for such $g$, because $\Delta g \cap \Delta = \varnothing$. 
Note also that $h^gh$ does not fix $x$: since $g\in G_x$, we have $xh^gh = xhgh = ygh$, so $h^gh$ fixes $x$ if and only if $yg$ equals $xh^2$, which equals $yh$, but $yg \neq yh$. 
Therefore, and since $x = xg \sim xhg = xgg^{-1}hg = xh^g$, Lemma~\ref{diamondOriginal}(ii) implies that $x,xh,xh^g$ all lie on a common line. 
That is, $x=xg$ and $xh^g=xhg$ both lie on $\lambda$, so $g$ fixes $\lambda$. 
The claim is proved. 
Now consider the elements $g_1,g_2,g_3 \in G_x$ given by
\begin{align*}
g_1 &:= (1\;5)(2\;6)(3\;4)(7\;8), \\
g_2 &:= (1\;7)(2\;8)(3\;4)(5\;6), \\
g_3 &:= (1\;5)(2\;6)(3\;7)(4\;8).
\end{align*}
Then $\Delta g_i \cap \Delta = \varnothing$ and $yg_i \neq yh$ for $i=1,2$ and $3$. 
(To check that $yg_i \neq yh$, observe that each $yg_i$ contains the partition class $B_1=\{1,2\}$ but that $yh$ does not contain $B_1$.) 
Hence each $g_i$ lies in $G_\lambda$ by the claim. 
Now consider the setwise stabiliser $(G_x)_{B_1}$ of $B_1$ in $G_x$. 
Then, in particular, $g_1$ does not lie in $(G_x)_{B_1}$, because $g_1$ maps $B_1$ to $B_3$. 
Since $(G_x)_{B_1}$ is a maximal subgroup of $G_x$, this implies that $G_x = \langle (G_x)_{B_1},g_1 \rangle$. 
We show that $(G_x)_{B_1}$ is contained in $G_\lambda$, which implies that $G_x$ is contained in $G_\lambda$, and hence, by maximality of $G_x$ in $G$, that $G_x = G_\lambda$. 
Let $a \in (G_x)_{B_1}$. 
Since all of the $yg_i$ contain $B_1$, it follows that all of the $y(g_ia)$ contain $B_1$, and hence that none of these elements is equal to $yh$. 
We now show that for $i$ equal to one of $1$, $2$, or $3$, we have $\Delta (g_ia) \cap \Delta = \varnothing$. 
By the claim, this implies that $g_ia \in G_\lambda$, and it follows that $a = g_i^{-1}(g_i a) \in G_\lambda$ as required. 
Since $a$ fixes $B_1=\{1,2\}$, we just need to choose $i$ such that $a$ does not map any of the elements of $\Delta g_i$ to the element $3$. 
If $3 \not \in \{4a,5a,6a\}$ then $\Delta(g_1a) = \{1,2,3\}(g_1a) = \{4a,5a,6a\}$ and hence $\Delta(g_1a)\cap \Delta = \varnothing$. 
If $3 \in \{5a,6a\}$ then $\Delta(g_2a) = \{4a,7a,8a\}$ and hence $\Delta(g_2a)\cap \Delta = \varnothing$. 
Finally, if $3=4a$ then $\Delta(g_3a) = \{5a,6a,7a\}$ and hence $\Delta(g_3a)\cap \Delta = \varnothing$. 

\item Now suppose that $k\ge 3$. 
We first claim that if $g\in G_x$ is an element satisfying $\Delta g \cap \Delta = \varnothing$, then $g\in G_\lambda$. 
To prove this, begin by observing that $h=(1\;2)(k\;k+1)$ commutes with $h^g$ for such $g$, and that $x = xg \sim xhg = xgg^{-1}hg = xh^g$. 
If $hh^g=h^gh$ does not fix $x$, then Lemma~\ref{diamondOriginal}(ii) implies that $x,xh,xh^g$ all lie on a common line. 
That is, both $x=xg$ and $xh^g=xhg$ lie on $\lambda$, and so $g$ fixes $\lambda$. 
If $hh^g=h^gh$ does fix $x$ then we have $x = xh^gh = (xg^{-1})hgh = (xh)gh = ygh$, and hence $y = xh = ygh^2 = yg$ (because $h^2=1$), so in this case $g$ fixes both $x$ and $y$, and hence also fixes $\lambda$. 
The claim is proved. 
Now consider the elements $g_1,g_2 \in G_x$ given by
\begin{align*}
g_1 &:= \begin{cases}
(1\; 2k+1)(2\; 2k+2)\cdots(k\;3k)(k+1\;k+2) & \text{if $k$ is odd} \\
(1\; 2k+1)(2\; 2k+2)\cdots(k\;3k)(k+1\;k+2)(k+3\;k+4) & \text{if $k$ is even}, 
\end{cases} \\
g_2 &:= \begin{cases}
(1\; 2k+1)(2\; 2k+2)\cdots(k\;3k)(k+1\;k+3) & \text{if $k$ is odd} \\
(1\; 2k+1)(2\; 2k+2)\cdots(k\;3k)(k+1\;k+3)(k+2\;k+4) & \text{if $k$ is even}.
\end{cases}
\end{align*}
Then $\Delta g_1 \cap \Delta = \Delta g_2 \cap \Delta = \varnothing$, and hence $g_1,g_2 \in G_\lambda$ by the claim. 
Consider the setwise stabiliser $(G_x)_{B_3}$ of $B_3$ in $G_x$. 
Since $g_1 \not \in (G_x)_{B_3}$, and since $(G_x)_{B_3}$ is a maximal subgroup of $G_x$, we have $\langle (G_x)_{B_3},g_1 \rangle = G_x$. 
We now show that $(G_x)_{B_3} \le G_\lambda$, which implies that $G_x \le G_\lambda$, and hence, by maximality of $G_x$ in $G$, that $G_x = G_\lambda$. 
Let $a \in (G_x)_{B_3}$. 
We show that for one of $i=1$ or $i=2$ we have $\Delta (g_ia) \cap \Delta = \varnothing$. 
By the claim, this implies that $g_ia \in G_\lambda$, and hence that $a = g_i^{-1} (g_ia) \in G_\lambda$, as required. 
We have $\Delta g_1 = \{2k+1,2k+2,3k,k+2\}$ and $\Delta g_2 = \{2k+1,2k+2,3k,k+3\}$. 
Since $a$ fixes $B_3$ setwise, we have $\Delta (g_1 a) \cap \Delta = \varnothing$ unless $(k+2)a = k+1$, and in this case we have instead $\Delta (g_2 a) \cap \Delta = \varnothing$. 
\end{enumerate}

\subsubsection{Primitive point stabiliser} \label{ss:prim}

Now suppose that the stabiliser $G_x$ of a point $x\in \mathcal{P}$ is a primitive subgroup of $S_n$ in its action on $\{1,\ldots,n\}$. 
We need the following lemma about the index of a primitive permutation group. 
The result is likely to be well known, but we have been unable to find a reference for it, so we include a proof.

\begin{lemma} \label{GxPrim}
Let $G=A_n$ or $S_n$, where $n \ge 9$. 
If $H$ is a primitive maximal subgroup of $G$ that does not contain $A_n$, then $|G:H|$ is divisible by $4$. 
\end{lemma}

\begin{proof}
Let $Q$ and $P$ be Sylow $2$-subgroups of $H$ and $S_n$, respectively, such that $Q \le P$. 
Write $Q^+ = Q\cap A_n$ and $P^+ = P\cap A_n$. 
Then $P^+$ is a Sylow $2$-subgroup of $A_n$, and $|P:P^+|=2$. 
If $G=S_n$ then $H$ is not contained in $A_n$, by maximality of $H$, so $|Q:Q^+|=2$ and hence $|G:H|_2 = |P:Q| = |P^+:Q^+|$, where $|G:H|_2$ is the largest power of $2$ dividing $|G:H|$. 
If $G=A_n$ then $P=P^+$ and $Q=Q^+$, and again $|G:H|_2 = |P:Q| = |P^+:Q^+|$.
We now show that the $2$-power $|P^+:Q^+|$ is at least $4$. 
If $|P^+:Q^+|=1$ then $Q^+=P^+$, so in particular $Q^+$ is a Sylow $2$-subgroup of $A_n$ and hence contains a double transposition (a product of two disjoint transpositions). 
However, this means that the primitive group $H$ contains a double transposition, and since $n\ge 9$, a theorem of Jordan~\cite{Jordan:1875lo} (see also \cite[Example~3.3.1]{Dixon:1996ud}) then implies that $H$ contains $A_n$, a contradiction. 
Now suppose that $|P^+:Q^+|=2$. 
Let $K=\langle (1\;2),(3\;4),(5\;6) \rangle$. 
By conjugating $P$ and  $H$ simultaneously (and conjugating $Q$ along with $H$), we may assume that $K\le P$. 
Let $K_0=K \cap A_n$. 
Then $|K|=8$, $|K_0|=4$, $K_0\le P^+$, and $K_0$ contains three double transpositions. 
If $K_0\le Q^+$ then $Q^+$ contains a double transposition, and we are done. 
Otherwise, $P^+=Q^+K_0$, and so $2=|P^+:Q^+|=|Q^+K_0:Q^+|=|K_0:Q^+\cap K_0|$. 
Thus $Q^+ \cap K_0$ is nontrivial, so $Q^+$ contains a double transposition, and hence so does $H$, a contradiction.
\end{proof}

We have $|\mathcal{P}| = |G:G_x|$ with $G_x$ a primitive maximal subgroup of $G$, and $G_x$ does not contain $A_n$ because this would imply that $|\mathcal{P}| = 1$ or $2$. 
Therefore, and since we are assuming that $n\ge 14 > 9$, Lemma~\ref{GxPrim} implies that $|\mathcal{P}|$ is divisible by $4$, and Lemma~\ref{stNotCoprime} together with properties (PH) and (PO) in Section~\ref{sec2} therefore implies that $\gcd(s,t)>1$. 
We note also that, since $G_x$ is a maximal subgroup of $A_n$ or $S_n$ and $n\ge 14$, $G_x$ is not cyclic of prime order and, in particular, the stabiliser $(G_x)_j$ of an element $j \in \{1,\ldots,n\}$ is nontrivial. 
Moreover, we note that by a theorem of Jordan~\cite{Jordan:1872kk,Manning:1909dw}, the primitive group $G_x$ contains no elements with cycle type $3^1$ or $3^2$ if $n\ge 10$ (and hence, in particular, for $n\ge 14$). 

Now consider the $3$-cycle $h = (1\; 2\; 3) \in G$. 
Then, as noted above, $h$ is fixed-point free in its action on $\mathcal{P}$, so Lemma~\ref{diamondOriginal}(i) implies that there exists $x \in \mathcal{P}$ such that $x \sim xh$. 
This implies that $xh \sim xh^2$, and hence that $xh^2 \sim xh^3 = x$. 
Since $\mathcal{S}$ contains no triangles, it follows that $h$ fixes the line $\lambda = \langle x,xh \rangle$. 
Write $\Delta = \{ 1,2,3 \}$. 
We first claim that if $g \in G_x$ satisfies $\Delta g \cap \Delta = \varnothing$, then $g \in G_\lambda$. 
To prove this, first note that $h$ and $h^g$ commute. 
Since $x \sim xh$, we have $x = xg \sim xhg = xgg^{-1}hg = xh^g$; that is, $x \sim xh^g$. 
Moreover, $hh^g$ is a permutation of type $3^2$ and is therefore fixed-point free on $\mathcal{P}$, as noted above. 
In particular, $hh^g$ does not fix $x$, so Lemma~\ref{diamondOriginal}(ii) implies that $x,xh,xh^g$ lie on a common line. 
That is, $x=xg$ and $xh^g = xhg$ both lie on $\lambda = \langle x,xh \rangle$, and hence $g$ fixes $\lambda$ as claimed. 
Let us assume at this point that $n>15$. 
We show that $(G_x)_1 \le G_\lambda$. 
Let $a \in (G_x)_1$ and write $\Gamma = \Delta \cup \{ 2a^{-1},3a^{-1} \}$. 
Suppose that $b \in G_x$ satisfies $\Delta b \cap \Gamma = \varnothing$. 
Then, in particular, $\Delta b \cap \Delta = \varnothing$, and moreover, $\Delta (ba) \cap \Delta = (\Delta b \cap \Delta a^{-1})a \subseteq (\Delta b \cap \Gamma)a = \varnothing$. 
Therefore, by the above argument, $b$ and $ba$ both lie in $G_\lambda$, and hence $a = b^{-1}(ba) \in G_\lambda$. 
By Neumann's Separation Lemma~\cite[Theorem~2]{Birch:1976ad}, such an element $b \in G_x$ exists if $n > |\Delta| |\Gamma| = 15$, and so $(G_x)_1 \le G_\lambda$ as claimed. 
By an analogous argument, $(G_x)_2 \le G_\lambda$, so $G_x = \langle (G_x)_1,(G_x)_2 \rangle \le G_\lambda$ (because $(G_x)_1$ is nontrivial and is a maximal subgroup of $G_x$) and hence $G_x = G_\lambda$ (because $G_x$ is a maximal subgroup of $G$). 
However, $h$ fixes $\lambda$ but not $x$, so this is a contradiction. 

It remains to consider the cases $n=14$ and $n=15$. 
We require a primitive maximal subgroup $H$ of $G=A_n$ or $S_n$ such that $H$ does not contain $A_n$, and such that $|G:H|$ is equal to the number of points of a (finite thick) generalised hexagon or octagon. 
For $n = 14$, we have two candidates: $\textsf{PSL}(2,13)$ and $\textsf{PGL}(2,13)$. 
The first is maximal in $A_{14}$, and the second is maximal in $S_{14}$. 
We have $| A_{14} : \textsf{PSL}(2,13) | =  | S_{14} : \textsf{PGL}(2,13) | = 39\;916\;800$, and one checks computationally that this cannot be the number of points of a generalised hexagon or octagon (using properties (PH) and (PO) in Section~\ref{sec2}). 
For $n = 15$, we have four candidates: $A_7$, $A_6$, $S_6$, and $\textsf{PSL}(4,2)$. 
All are contained in $A_{15}$, and only $\textsf{PSL}(4,2)$ is maximal in $A_{15}$ (it contains the other three). 
We have $| A_{15} : \textsf{PSL(4,2)} | = 32\;432\;400$, which also cannot be the number of points of a generalised hexagon or octagon.

\subsection{Proof of Theorem~\ref{thm:main} and Corollary~\ref{mainCorollary}} \label{s:summary}

Let us summarise the proof of Theorem~\ref{thm:main} and deduce Corollary~\ref{mainCorollary}. 
Theorem~\ref{thm:main} asserts that if Hypothesis~\ref{hyp} holds, then $G$ must be an almost simple group of Lie type. 
Schneider and Van Maldgehem~\cite[Lemma~4.2(i)]{Schneider:2008lr} already proved that $G$ cannot have O'Nan--Scott type HA, HS, or HC under Hypothesis~\ref{hyp}, and Lemmas~\ref{lemma:PACD} and~\ref{lemma:SDTW} further imply that $G$ cannot have type PA, CD, SD, or TW. 
Therefore, $G$ must be an almost simple group. 
By a result of Buekenhout and Van Maldeghem~\cite{Buekenhout:1993bf}, the socle of $G$ cannot be a sporadic group, and by Lemma~\ref{lemma:An}, it cannot be an alternating group. 
Therefore, the only remaining possibility is that $\operatorname{soc}(G)$ is a simple group of Lie type, namely that $G$ is an almost simple group of Lie type. 

Now, the conclusion of Theorem~\ref{thm:main} holds in particular if $G$ also acts primitively on lines, so case~(i) of Corollary~\ref{mainCorollary} is the only possibility for $d\in \{6,8\}$. 
If $d=3$ then, by the work of Gill~\cite{Gill:2007xw}, the only other possibility is case~(ii). 
Finally, suppose that $d=4$. 
Bamberg et~al.~\cite{Bamberg:2012yf} proved, in the first instance, that $G$ must be an almost simple group with $\operatorname{soc}(G)=A_n$, $n\ge 5$, if it acts primitively on both the points and lines of $\mathcal{S}$. 
They then showed, using only the assumption of point-primitivity, that the stabiliser $G_x$ of a point of $\mathcal{S}$ must be a primitive subgroup of $S_n$ in the natural action on $\{1,\ldots,n\}$ (see their Sections~5.1 and~5.2). 
By duality, the same is true for lines under the assumption of line-primitivity. 
In particular, the automorphism $h=(1\;2\;3)$ cannot fix any point or line of $\mathcal{S}$, because this would force the stabiliser to contain $A_n$ and hence imply that $\mathcal{S}$ has either at most two points or at most two lines. 
However, if $\gcd(s,t)>1$ then this contradicts \cite[Lemma~3.4]{Bamberg:2012yf}, which says that an automorphism of order $2$ or $3$ must fix either a point or a line. 
Therefore, the only remaining possibility is that $\gcd(s,t)=1$, as in~(iii).

\bibliographystyle{abbrv}
\bibliography{HexOct13}

\begin{thebibliography}{10}

\bibitem{Bamberg:2012yf}
J.~Bamberg, M.~Giudici, J.~Morris, G.~F. Royle, and P.~Spiga.
\newblock Generalised quadrangles with a group of automorphisms acting
  primitively on points and lines.
\newblock {\em J. Combin. Theory Ser. A}, 119(7):1479--1499, 2012.

\bibitem{Betten:2003vl}
A.~Betten, A.~Delandtsheer, A.~C. Niemeyer, and C.~E. Praeger.
\newblock On a theorem of {W}ielandt for finite primitive permutation groups.
\newblock {\em J. Group Theory}, 6(4):415--420, 2003.

\bibitem{Birch:1976ad}
B.~J. Birch, R.~G. Burns, S.~O. Macdonald, and P.~M. Neumann.
\newblock On the orbit-sizes of permutation groups containing elements
  separating finite subsets.
\newblock {\em Bull. Austral. Math. Soc.}, 14(1):7--10, 1976.

\bibitem{Buekenhout:1993bf}
F.~Buekenhout and H.~Van~Maldeghem.
\newblock Remarks on finite generalized hexagons and octagons with a
  point-transitive automorphism group.
\newblock In {\em Finite geometry and combinatorics ({D}einze, 1992)}, volume
  191 of {\em London Math. Soc. Lecture Note Ser.}, pages 89--102. Cambridge
  Univ. Press, Cambridge, 1993.

\bibitem{Buekenhout:1994zp}
F.~Buekenhout and H.~Van~Maldeghem.
\newblock Finite distance-transitive generalized polygons.
\newblock {\em Geom. Dedicata}, 52(1):41--51, 1994.

\bibitem{Dembowski:1997sy}
P.~Dembowski.
\newblock {\em Finite geometries}.
\newblock Classics in Mathematics. Springer-Verlag, Berlin, 1997.
\newblock Reprint of the 1968 original.

\bibitem{Dixon:1996ud}
J.~D. Dixon and B.~Mortimer.
\newblock {\em Permutation groups}, volume 163 of {\em Graduate Texts in
  Mathematics}.
\newblock Springer-Verlag, New York, 1996.

\bibitem{Feit:1964os}
W.~Feit and G.~Higman.
\newblock The nonexistence of certain generalized polygons.
\newblock {\em J. Algebra}, 1:114--131, 1964.

\bibitem{MR0166261}
W.~Feit and J.~G. Thompson.
\newblock Solvability of groups of odd order.
\newblock {\em Pacific J. Math.}, 13:775--1029, 1963.

\bibitem{Gill:ft}
N.~Gill.
\newblock Transitive projective planes and insoluble groups.
\newblock {\em To appear in Trans. Amer. Math. Soc.}

\bibitem{Gill:2007xw}
N.~Gill.
\newblock Transitive projective planes.
\newblock {\em Adv. Geom.}, 7(4):475--528, 2007.

\bibitem{Jordan:1872kk}
C.~Jordan.
\newblock Sur la limite de transitivit{\'e} des groupes non altern{\'e}s.
\newblock {\em Bull. Soc. Math. France}, 1:40--71, 1872/73.

\bibitem{Jordan:1875lo}
C.~Jordan.
\newblock Sur la limite du degr\'e des groupes primitifs qui contiennent une
  subsitution donn\'ee.
\newblock {\em J. Reine Angew. Math.}, 79(248--258), 1875.

\bibitem{Kantor:1987qq}
W.~M. Kantor.
\newblock Primitive permutation groups of odd degree, and an application to
  finite projective planes.
\newblock {\em J. Algebra}, 106(1):15--45, 1987.

\bibitem{Manning:1909dw}
W.~A. Manning.
\newblock On the order of primitive groups.
\newblock {\em Trans. Amer. Math. Soc.}, 10(2):247--258, 1909.

\bibitem{Parkinson}
J.~Parkinson, B.~Temmermans, and H.~Van~Maldeghem.
\newblock The combinatorics of automorphisms and opposition in generalised
  polygons.
\newblock {\em To appear in Ann. Combin.}

\bibitem{MR1477745}
C.~E. Praeger.
\newblock Finite quasiprimitive graphs.
\newblock In {\em Surveys in combinatorics, 1997 ({L}ondon)}, volume 241 of
  {\em London Math. Soc. Lecture Note Ser.}, pages 65--85. Cambridge Univ.
  Press, Cambridge, 1997.

\bibitem{Schneider:2008lr}
C.~Schneider and H.~Van~Maldeghem.
\newblock Primitive flag-transitive generalized hexagons and octagons.
\newblock {\em J. Combin. Theory Ser. A}, 115(8):1436--1455, 2008.

\bibitem{MR1675918}
A.~E. Schroth.
\newblock How to draw a hexagon.
\newblock {\em Discrete Math.}, 199(1-3):161--171, 1999.

\bibitem{Temmermans:2009zm}
B.~Temmermans, J.~A. Thas, and H.~Van~Maldeghem.
\newblock On collineations and dualities of finite generalized polygons.
\newblock {\em Combinatorica}, 29(5):569--594, 2009.

\bibitem{Thas:2008px}
K.~Thas and D.~Zagier.
\newblock Finite projective planes, {F}ermat curves, and {G}aussian periods.
\newblock {\em J. Eur. Math. Soc. (JEMS)}, 10(1):173--190, 2008.

\bibitem{Tits:1959cl}
J.~Tits.
\newblock Sur la trialit{\'e} et certains groupes qui s'en d{\'e}duisent.
\newblock {\em Inst. Hautes {\'E}tudes Sci. Publ. Math.}, (2):13--60, 1959.

\bibitem{Van-Maldeghem:1998sj}
H.~Van~Maldeghem.
\newblock {\em Generalized polygons}.
\newblock Modern Birkh{\"a}user Classics. Birkh{\"a}user/Springer Basel AG,
  Basel, 1998.
\newblock [2011 reprint of the 1998 original] [MR1725957].

\end{thebibliography}

\end{document}